\documentclass[12pt]{amsart}
\usepackage{anysize}
\marginsize{2.5cm}{2.5cm}{2.5cm}{2.5cm}

\usepackage{lineno,hyperref}
\modulolinenumbers[5]

\usepackage[T1]{fontenc}
\usepackage{textpos}

\usepackage{amsmath,amsthm,amscd,amsfonts,amssymb,enumerate,pifont}
\usepackage{graphicx}		
\usepackage{color}

\usepackage{ulem}
\newcommand{\del}[1]{}
\newcommand{\add}[1]{#1}

\renewcommand{\emph}[1]{\textit{#1}}

\usepackage{mathtools}

\newtheorem{theorem}{Theorem}[section]
\newtheorem{lemma}[theorem]{Lemma}
\newtheorem{corollary}[theorem]{Corollary}
\numberwithin{equation}{section}
\newcommand{\lby}[1]{(\refstepcounter{obs}\theobs\label{#1})}
\newcommand{\lbn}{(\refstepcounter{obs}\theobs)}
\newcounter{obs}
\usepackage[all]{xy}

\begin{document}

	
\title[Duality in non-abelian algebra IV]{Duality in non-abelian algebra IV.\\ Duality for groups and a universal\\ isomorphism theorem}

\author{Amartya Goswami}
\address{Department of Mathematics and Applied Mathematics, University of Limpopo, Private Bag X1106, Sovenga 0727, South Africa}

\author{Zurab Janelidze}
\address{Mathematics Division, Department of Mathematical Sciences, Stellenbosch University, Private Bag X1, Matieland 7602, South Africa}

\maketitle

\begin{abstract} 
Abelian categories provide a self-dual axiomatic context for establishing homomorphism theorems (such as the isomorphism theorems and homological diagram lemmas) for abelian groups, and more generally, modules. In this paper we describe a self-dual context which allows one to establish the same theorems in the case of non-abelian group-like structures; the question of whether such a context can be found has been left open for \add{seventy} years. We also formulate and prove in our context a \emph{universal isomorphism theorem} from which all other isomorphism theorems can be deduced. 
\end{abstract}



\section*{Introduction}

The work presented in this paper is a culminating point and gives a major application of the ``theory of forms'' that was being developed in earlier papers \cite{JW14, JW16, JW16b} from the same series. The present paper is, however, entirely self-contained and can be read independently of the previous ones.

\medskip

Since as early as the 1930's, it has been noticed that certain theorems of group theory (and in particular, the isomorphism theorems) can be proved using the structure of subgroup lattices, not referring to elements of groups.  Mac~Lane in \cite{M50,M48} notices further that a number of remarkable duality phenomena arise for groups, once basic group-theoretic notions referring to elements of groups are reformulated purely in the language of group homomorphisms and their composition. He then proposes a general axiomatic context for exploiting such phenomena, focusing however, almost from start, on abelian groups; quoting from \cite{M50}:
\begin{itemize}
	\item[]\emph{A further development giving the first and second isomorphism theorems, and so on, can be made by introducing additional carefully chosen dual axioms. This will be done below only in the more symmetrical abelian case.}
\end{itemize}
The consideration of the ``abelian case'' led to the notion of an \emph{abelian category}, refined by Buchsbaum in \cite{B55}, which became a central notion in homological algebra particularly after the work \cite{G57} of Grothendieck. Abelian categories are defined by dual axioms (a category is abelian if and only if its dual category is abelian) and they allow a convenient unified treatment of homomorphism theorems for abelian groups, and more generally, for modules over a ring. Duality is a practical tool in such a treatment, since it \add{allows} to get two dual results out of one. The problem of finding a similar self-dual treatment of non-abelian groups and other group-like structures was taken up by Wyler in \cite{W66,W71}. However, he only gave, in some sense partial solutions to this problem. The axioms in \cite{W66} are self-dual, but too restrictive to include all group-like structures (his axiom BC2 would force every epimorphism to be normal, which is true for groups, but not true, for instance, for rings), while the axioms in \cite{W71} are not self-dual (and also, are in some sense too general, so the isomorphism theorems are obtained only up to a suitable reformulation). No further progress on this problem has since been made. Instead, many investigations have been carried out around the non-dual axioms (by ``non-dual'' we mean ``non-self-dual''). Today, the notion of a \emph{semi-abelian category}, introduced in \cite{JMT02}, is considered as a ``non-commutative'' analogue of the notion of an abelian category, which includes almost all desirable categories of group-like structures (groups, non-unitary rings, Lie algebras, \add{Jordan algebras,} etc.)~and others (e.g.~loops and Heyting meet semi-lattices) in addition to abelian categories. The context of a semi-abelian category allows a unified treatment of all standard homomorphism theorems for these structures (see e.g.~\cite{BB04}). This notion has in fact deep roots in several parallel developments in abstract algebra, including the above-mentioned works of Mac~Lane and Wyler, but also the works of Barr~\cite{B71}, Bourn~\cite{B91}, and Ursini~\cite{U73} should be mentioned. Unlike the axioms of an abelian category, however, those of a semi-abelian category are highly non-dual; a semi-abelian category whose dual is also semi-abelian is already an abelian category! 

\add{In this paper we propose an entirely self-dual treatment of homomorphism theorems for a wide range of algebraic structures that includes essentially all ``group-like'' structures; this could be considered as achieving the goal originally set out by Mac~Lane seventy years ago.} \add{The  underlying axiomatic context can be seen as a self-dual adaptation of the axiomatic context of a semi-abelian category. However, our context includes not only that of a semi-abelian category as a special case, but also the context of a generalized exact category in the sense of Grandis \cite{G92} and hence adds (categories of) projective spaces and graded abelian groups to the list of examples; furthermore, our theory also includes examples not covered by these two classes of categories, such as the categories of rings with identity and Boolean algebras.} 

Our self-dual axioms \add{are} reformulations of the ``first isomorphism theorem'', the ``lattice isomorphism theorem'' and stability of normal subgroups under joins. This is in some sense not too surprising, neither from the point of view of classical exposition of isomorphism theorems such as the one found in \cite{MB99}, nor from the point of view of the theory of semi-abelian categories. In our theory, we incorporate and thoroughly exploit Mac~Lane's method of chasing subobjects (that he developed for abelian categories in \cite{M63}), which allows us to \add{discover} a single isomorphism theorem, the so-called ``universal isomorphism theorem'', for deducing all others. In particular, as an illustration, we use it to establish two of Noether's well-known isomorphism theorems, obtained originally for modules in \cite{N27}, an isomorphism theorem from \cite{MB99} and the ``butterfly lemma'' due to Zassenhaus \cite{Z34}. Other isomorphism theorems as well as the diagram lemmas of homological algebra can be established similarly, using the methods presented in this paper. In particular, our ``homomorphism induction theorem'', of which the universal isomorphism theorem is a consequence, is suitable for constructing the connecting homomorphisms in homological algebra. 

We should mention that the category-theoretic investigations in the predecessors \cite{JW14, JW16, JW16b} of this paper have played a crucial role in arriving to the list of axioms given in the present paper. Many of the basic consequences of these axioms have already been worked out in \cite{JW14, JW16, JW16b}, however, we do not refer to them in this paper, not to overload the text with citations. \add{Moreover, in this paper we avoid the discussion of the category-theoretic aspects of our work, aside of the brief remarks in the concluding section of the paper, and instead present the self-dual approach to homomorphisms theorems of groups in a primitive language, as a fragment of ordinary group theory. We take this approach for two reasons: to keep the statements of homomorphism theorems as close as possible to their original formulation for groups, and to ensure that the paper is accessible to a wide readership.}

\section{The language}

Our theory is presented in a (first-order) language derived just from groups, group homomorphisms and their composition, subgroups of groups and inclusion of subgroups, as well as direct and inverse images of subgroups along group homomorphisms. Similar to duality in projective plane geometry, in a lattice or in a category, there is inherent duality in our language under which the system of axioms that we will use is invariant. This means of course that the whole theory is invariant under duality, or in other words, it is a self-dual theory. 

The language of our theory consists of the following atomic expressions, which will have their usual meaning in group theory:
\begin{itemize}
	\item $G$ is a group;
	
	\item $f$ is a homomorphism $f\colon X\to Y$ from the group $X$ to the group $Y$ (i.e.~$X$ is the domain of $f$ and $Y$ is the codomain of $f$);
	
	\item $h$ is the composite $h=gf$ of homomorphisms $f\colon X\to Y$ and $g\colon Y\to Z$;
	
	\item $S$ is a subgroup of a group $G$;
	
	\item there is an inclusion $S\subseteq T$ between two subgroups $S$ and $T$ of $G$;
	
	\item the subgroup $B$ of the group $Y$ is the direct image $B=fA$ of the subgroup $A$ of the group $X$ under the homomorphism $f\colon X\to Y;$
	
	\item the subgroup $A$ of the group $X$ is the inverse image $A=f^{-1}B$ of the subgroup $B$ of the group $Y$ under the homomorphism $f\colon X\to Y.$
\end{itemize}
The duals of these expressions are defined as, respectively:
\begin{itemize}
	\item $G$ is a group;
	
	\item $f$ is a homomorphism $f\colon Y\to X$ from the group $Y$ to the group $X$ (i.e.~$Y$ is the domain of $f$ and $X$ is the codomain of $f$);
	
	\item $h$ is the composite $h=fg$ of homomorphisms $g\colon Z\to Y$ and $f\colon Y\to X;$
	
	\item $S$ is a subgroup of a group $G;$
	
	\item there is an inclusion $T\subseteq S$ between two subgroups $S$ and $T$ of $G$;
	
	\item the subgroup $B$ of the group $Y$ is the inverse image $B=f^{-1}A$ of the subgroup $A$ of the group $X$ under the homomorphism $f\colon Y\to X$;
	
	\item the subgroup $A$ of the group $X$ is the direct image $A=fB$ of the subgroup $B$ of the group $Y$ under the homomorphism $f\colon Y\to X.$
\end{itemize}
The dual of an expression formed from these atomic expressions is obtained by replacing each atomic expression in it by its dual. Notice that the ``double-dual'' of an expression (i.e.~dual of the dual expression) is the original expression; in other words, if an expression $\varphi$ is dual to an expression $\psi,$ then $\psi$ is dual to $\varphi.$

Of course, many group-theoretic concepts can be defined in this language. Let us give some examples:
\begin{itemize}
	\item The identity homomorphism $1_G\colon G\to G$ for a group $G,$ is the homomorphism such that $1_Gf=f$ and $g1_G=g$ for arbitrary homomorphisms $f\colon F\to G$ and $g:G\to H.$
	
	\item An isomorphism is a homomorphism $f\colon X\to Y$ such that $fg=1_Y$ and $gf=1_X$ for some homomorphism $g\colon Y\to X.$
	
	\item The trivial subgroup of a group $G$ is the subgroup $1$ of $G$ such that we have $1\subseteq S$ for any other subgroup $S$ of $G.$
	
	\item The image of a group homomorphism $f\colon X\to Y,$ written as $\mathsf{Im}f,$ is the direct image of the largest subgroup of $X$ under $f.$
	
	\item The kernel of a group homomorphism $f\colon X\to Y$, written as $\mathsf{Ker}f,$ is the inverse image of the trivial subgroup of $Y.$
	
	\item A normal subgroup of a group $G$ is its subgroup $S$ which is the kernel of some group homomorphism $f\colon G\to H.$
\end{itemize} 
The dual of a concept/property (or a construction, theorem, etc.) is one given by taking the dual of the defining/determining expression. In some cases this results in the same/equivalent concept/property, or the same construction, in which case we say that the concept/property/construction is \emph{self-dual}. For example, such is the construction of the identity homomorphism $1_G$ from a group $G.$ The concept of an isomorphism is also self-dual. Dual of the concept of a trivial subgroup of a group $G$ is the concept of the largest subgroup of $G,$ which in group theory is given by the same $G.$ The concepts of image and kernel of a homomorphism are dual to each other. Dual of the concept of a normal subgroup will be referred to as a \emph{conormal subgroup}. A conormal subgroup $S$ of a group $G$ is a subgroup of $G$ which appears as the image of some group homomorphism $f\colon F\to G.$ In group theory all subgroups are of course conormal, but in our axiomatic development we will not require this since its dual requirement for all subgroups to be normal fails for ordinary groups.

\add{In this paper we do not attempt to describe what the most convenient logical formalism for our theory should be, although this line of research would certainly be an interesting one, especially in the direction of developing an efficient proof-assisting software for our theory. As it stands, we can assume that we are working within (in fact, intuitionistic) first-order logic with predicate symbols necessary to express the statements in the first two bullet lists above. It is worth remarking, however, that among logical operations we only use conjunction, implication, the ``truth constant'', and existential and universal quantifiers --- we will never use either disjunction or the ``false constant''.}

\section{The axioms and their immediate consequences}
We now state the axioms for our theory, and discuss their basic consequences. Each of these axioms is self-dual in the sense explained in the previous section.
\begin{description}
	\item[Axiom 1.1] Composition of homomorphisms is associative, i.e.~we have
	$$(fg)h=f(gh)$$ whenever the three homomorphisms $f,g,h$ are arranged as
	$$\xymatrix{V\ar[r]^-{h} & W\ar[r]^-{g} & X\ar[r]^-{f} & Y }$$
	and each group has an identity homomorphism.
	
	\item[Axiom 1.2] For each group $G,$ subgroup inclusion is reflexive, transitive and antisymmetric --- for any three subgroups $A,$ $B$ and $C$ of a group $G,$ we have: $A\subseteq A$; if $A\subseteq B$ and $B\subseteq C$ then $A\subseteq C$; and, if $A\subseteq B$ and $B\subseteq A,$ then $A=B.$
	
	\item[Axiom 1.3] For each homomorphism $f\colon X\to Y$ and for any subgroup $A$ of $X$ and $B$ of $Y$ we have 
	$$f(A)\subseteq B\quad\Leftrightarrow\quad A\subseteq f^{-1}(B).$$
	
	\item[Axiom 1.4] The direct and inverse image maps for the identity homomorphisms are identity maps, and for any two homomorphisms $$\xymatrix{W\ar[r]^-{g} & X\ar[r]^-{f} & Y }$$
	and subgroups $A$ of $W$ and $B$ of $Y,$ we have $$f(gA)=(fg)A\text{ and }g^{-1}(f^{-1}B)=(fg)^{-1}B.$$
\end{description}
These axioms can be conveniently put together into one axiom using basic concepts of category theory as follows:
\begin{description}
	\item[Axiom 1] Groups and group homomorphisms, under composition of homomorphisms, form a category (called the \emph{category of groups}). Furthermore, for each group $G,$ the subgroups of $G$ together with subgroup inclusions form a poset $\mathsf{Sub}G$; for each homomorphism $f\colon X\to Y$ the direct and inverse image maps form a monotone Galois connection 
	$$\xymatrix{ \mathsf{Sub}X\ar@<2pt>[r] & \mathsf{Sub}Y\ar@<2pt>[l]}$$ and this defines a functor from the category of groups to the category of posets and Galois connections.  
\end{description}

As usual for associative operations, associativity of composition allows us to write composites without the use of brackets. Axioms 1.2 and 1.3 have several useful consequences which come out from the general theory of Galois connections. 

Firstly, 
\begin{itemize}
	\item[\lbn] both direct and inverse image maps are monotone, i.e.~they preserve inclusion of subgroups. 
\end{itemize}
Thanks to duality, we only need to show that one of them does. This will be done by first showing that 
\begin{itemize}
	\item[\lby{obsS}] $A\subseteq f^{-1}fA$ whenever $fA$ is defined. 
\end{itemize}
Indeed, by Axiom 1.3 this is equivalent to $fA\subseteq fA$ which we do have by Axiom 1.2 (and dually, we always have $ff^{-1}B\subseteq B$). Now, if $A_1\subseteq A_2$ then using what we just proved for $A_2$ and Axiom 1.2, we conclude $A_1\subseteq f^{-1}fA_2.$ By Axiom 1.3, this gives $fA_1\subseteq fA_2,$ which proves that $f$ is monotone. 

From the above we also easily get:
\begin{itemize}
	\item[\lbn] $ff^{-1}fA=fA$ and $f^{-1}ff^{-1}B=f^{-1}B$ whenever $fA$ and $f^{-1}B$ are defined.
\end{itemize}

Furthermore, 
\begin{itemize}
	\item[\lbn] the direct image map will always preserve joins of subgroups and the inverse image map will always preserve meets of subgroups. 
\end{itemize}
Indeed, given a set $\mathcal{S}$ of subgroups of a group $G,$ its join is defined as the smallest subgroup $T$ of $G$ such that $S\subseteq T$ for all $S\in\mathcal{S}.$ Let us write $T=\bigvee\mathcal{S}.$ We want to show that when the join of $\mathcal{S}$ exists, the join of the set $f\mathcal{S}=\{fS\mid S\in\mathcal{S}\}$ of subgroups of the codomain of $f$ will exist and will be equal to $f\bigvee\mathcal{S}.$ For this, first observe that by monotonicity of the direct image map, we have $S'\subseteq f\bigvee\mathcal{S}$ for all $S'\in f\mathcal{S}.$ Suppose $S'\subseteq T$ for all $S'\in f\mathcal{S}.$ Then for each $S\in\mathcal{S}$ we have $S\subseteq f^{-1}T$ by Axiom 1.3. This gives $\bigvee\mathcal{S}\subseteq f^{-1}T,$ and using again Axiom 1.3 we obtain $f\bigvee\mathcal{S}\subseteq T.$ This proves that $f\bigvee\mathcal{S}$ is the join of $f\mathcal{S}.$ In the special case when $\mathcal{S}$ is the empty set, we obtain that 
\begin{itemize}
	\item[\lbn] the direct image of the trivial subgroup (when it exists) under a homomorphism is a trivial subgroup. 
\end{itemize}
Meets of sets of subgroups is defined dually, and by duality, we get at once that the inverse image maps preserve them. In particular, 
\begin{itemize}
	\item[\lbn] the inverse image of the largest subgroup (when it exists) is the largest subgroup.
\end{itemize} 
Axiom 1.4 has the following useful consequence:
\begin{itemize}
	\item[\lby{obsB}] the direct and the inverse image maps of an isomorphism $f$ are the same as the inverse and the direct image maps, respectively, of its inverse $f^{-1}.$
\end{itemize}

The next axiom will require existence of finite meets and joins, including the empty ones, which amounts to saying that the posets of subgroups are bounded lattices in the sense of order theory. In addition, we will require a special property of how the composites of direct and inverse images maps for a single homomorphism can be computed using the meets and joins. 

For any group homomorphism $f\colon X\to Y$ we have
$$ff^{-1}B=B\wedge \mathsf{Im}f$$
for any subgroup $B$ of $Y,$ where ``$\wedge$'' is the intersection of subgroups, which is the same as the meet of subgroups in the lattice of subgroups. This fact comes from a similar fact which is true for any map $f$ between sets and any subset $B$ of the codomain of $f.$ The dual of this fact is no longer true for sets, but is still true for groups:
$$f^{-1}fA=A\vee \mathsf{Ker}f$$
for any group homomorphism $f\colon X\to Y$ and subgroup $A$ of $X.$ Note that here ``$\vee$'' denotes the join of subgroups in the lattice of subgroups. This brings us to the second axiom:

\begin{description}
	\item[Axiom 2] The poset of subgroups of each group is a bounded lattice (each group has largest and smallest subgroups, and join and meet of two subgroups of a group always exist). Furthermore, for any group homomorphism $f\colon X\to Y$ and subgroups $A$ of $X$ and $B$ of $Y$ we have $ff^{-1}B=B\wedge \mathsf{Im}f$ and $f^{-1}fA=A\vee \mathsf{Ker}f.$
\end{description} 

The second part of this axiom has an alternative formulation (which gives the ``lattice isomorphism theorem'' in group theory): 
\begin{itemize}
	\item[\lby{obsY}] $A$ satisfies $A=f^{-1}fA$ if and only if $\mathsf{Ker}f\subseteq A$ and $B$ satisfies $B=ff^{-1}B$ if and only if $B\subseteq \mathsf{Im}f.$
\end{itemize}
It is not difficult to notice how this reformulation can be deduced from Axiom~2. Conversely, if this reformulation holds, then $B\wedge \mathsf{Im}f=ff^{-1}(B\wedge \mathsf{Im}f)=f(f^{-1}B\wedge f^{-1}\mathsf{Im}f)=ff^{-1}B,$ where the last equality follows from the fact that $f^{-1}\mathsf{Im}f$ must be the largest subgroup of $X.$

Given a subgroup $S$ of a group $G,$ we may consider $S$ as a group in its own right. The subgroup inclusion homomorphism $\iota_S\colon S\to G$ has the following universal property: $\mathsf{Im}\iota_S\subseteq S$ and for any other group homomorphism $f\colon F\to G$ such that $\mathsf{Im}f\subseteq S,$ there exists a unique homomorphism $u\colon F\to S$ which makes the triangle
$$\xymatrix{ S\ar[r]^-{\iota_S} & G \\ F\ar[ur]_-{f}\ar@{..>}[u]^-{u} & }$$
commute. This property determines the subgroup inclusion homomorphism uniquely up to an isomorphism in the sense that homomorphisms which have the same property are precisely all composites $\iota_S u$ where $u$ is an isomorphism. Such composites are of course precisely the injective homomorphisms. In our axiomatic theory we call these homomorphisms \emph{embeddings}. The dual notion is that of a \emph{projection}. In group theory these are the same as surjective group homomorphisms. Given a \add{normal} subgroup $S$ of a group $G,$ the homomorphisms having the dual universal property to the one above turn out to be precisely the surjective group homomorphisms whose kernel is $S.$ \add{This leads us to the following axiom:}

\begin{description}
	\item[Axiom 3] Each \add{conormal} subgroup $S$ of a group $G$ admits an embedding \add{$\iota_S\colon S/1\to G$ such that $\mathsf{Im}\iota_S\subseteq S$ and for arbitrary group homomorphism $f\colon U\to G$ such that $\mathsf{Im}f\subseteq S,$ we have $f=\iota_Su$ for a unique homomorphism $u\colon U\to S/1$. Dually, each normal subgroup $S$ of a group $G$ admits a projection $\pi_S\colon G\to G/S$ such that $S\subseteq\mathsf{Ker}\pi_S$ and for an arbitrary group homomorphism $g\colon G\to V$ such that $S\subseteq\mathsf{Ker}g,$ we have $g=v\pi_S$ for a unique homomorphism $v\colon G/S\to V.$}  
\end{description} 

Given a subgroup $S$ of a group $G,$ a homomorphism $m\colon M\to G$ is said be an \emph{embedding} associated to $S$ when $\mathsf{Im}m\subseteq S$ and for any homomorphism $m'\colon M'\to G$ such that $\mathsf{Im}m'\subseteq S,$ there exists a unique homomorphism $u\colon M'\to M$ such that $m'=mu.$ It is not difficult to verify that under the previous three axioms embeddings will have the following properties:
\begin{itemize}
	\item[\lby{obsE'}] Any embedding is a monomorphism, i.e.~if $mu=mu'$ then $u=u',$ for any embedding $m\colon M\to G$ and any pair of parallel homomorphisms $u$ and $u'$ with codomain $M.$
	
	\item[\lby{obsX}] For any two embeddings $m$ and $m'$ corresponding to a subgroup $S,$ we have $m'=mi$ for a suitable (unique) isomorphism $i.$ Moreover, if $m\colon M\to G$ is an embedding for $S$ then so is $mi,$ for any isomorphism $i$ with codomain $M.$ In particular, embeddings associated to a \add{conormal} subgroup $S$ of a group $G$ are precisely the composites $\iota_Si$ where $\iota_S$ is from Axiom 3 and $i$ is any isomorphism with codomain \add{$S/1.$} 
	
	\item[\lbn] Any embedding is an embedding associated to its image. Moreover, the image of an embedding associated to a \add{conormal} subgroup $S$ is the subgroup $S$\add{.}
	
	\item[\lbn] Any isomorphism is an embedding and an embedding is an isomorphism if and only if it is an embedding associated to the largest subgroup (of its codomain).     
\end{itemize} 
Dual properties hold for \add{the dual notion of a \emph{projection} associated with a subgroup.} \add{These properties imply that the identity homomorphism $1_G$ is an embedding associated with the largest subgroup of $G$, and at the same time, the projection associated with the smallest subgroup of $G$. We then adopt the following convention:} 
\begin{itemize}
	\item[\add{\lby{obsAF}}] \add{when $S$ is the largest subgroup of $G,$ the homomorphism $\iota_S$ is chosen to be specifically the identity homomorphism $1_G$ and not just any embedding associated with $S$; dually, $\pi_S=1_G$ when $S$ is the trivial subgroup of $G$.}
	\end{itemize}

A subgroup $B$ of a group $G$ is said to be \emph{normal to} a subgroup $A$ of $G$ when $B\subseteq A,$ and moreover, $A$ is conormal and $\iota_A^{-1}B$ is normal. As in group theory, we write $B\triangleleft A$ when this relation holds. This relation is preserved under inverse images:
\begin{itemize}
	\item[\lby{obsZ}] For any group homomorphism $f\colon X\to Y$ and subgroups $A$ and $B$ of $Y,$ if $B\vartriangleleft A$ and $f^{-1}A$ is conormal, then $f^{-1}B\vartriangleleft f^{-1}A.$   
\end{itemize}
This follows easily from the fact that normal subgroups are stable under inverse images along arbitrary homomorphisms, which in turn is a direct consequence of the definition of a normal subgroup as an inverse image of a trivial subgroup. \add{Dually, when $B$ is normal, $B\subseteq A$ and $\pi_BA$ is conormal, we say that $A$ is \emph{conormal to} $B.$ In group theory this simply means that $B$ is a normal subgroup of $G$ and $A$ is any subgroup of $G$ larger than $B.$} The dual result to normal subgroups being stable under inverse images is: 
\begin{itemize}
	\item[\lby{obsV}] conormal subgroups are stable under direct images along arbitrary homomorphisms. 
\end{itemize}

\begin{description}
	\item[Axiom 4] Any group homomorphism $f\colon X\to Y$ factorizes as $$f=\iota_{\mathsf{Im}f}h\pi_{\mathsf{Ker}f}$$ where $h$ is an isomorphism.
\end{description} 

The isomorphism $h$ in this axiom is precisely the one in the ``first isomorphism theorem'' in group theory, which states that the image of a group homomorphism is isomorphic to the quotient of the domain by the kernel:  
$$
\xymatrix{ X\ar[r]^{f}\ar[d]_{\pi_{\mathsf{Ker}f}} & Y\\
	X/\mathsf{Ker}f \ar[r]_{h} & \mathsf{Im}f/1\ar[u]_{\iota_{\mathsf{Im}f}}
}
$$
One crucial consequence of this axiom is that any 
\begin{itemize}
	\item[\lby{obsE}] embedding has trivial kernel and dually, the image of a projection is the largest subgroup of its codomain. 
\end{itemize}
Indeed, consider an embedding $m\colon M\to X.$ By the above axiom, it factorizes as $$m=\iota_{\mathsf{Im}m}h\pi_{\mathsf{Ker}m}$$ where $h$ is an isomorphism. Since $m$ is an embedding, it must be an embedding associated to its image. This forces the composite $h\pi_{\mathsf{Ker}m}$ to be an isomorphism. But then, since $h$ is also an isomorphism, we get that $\pi_{\mathsf{Ker}m}$ is an isomorphism. This implies that $\mathsf{Ker}m$ is the trivial subgroup of $M,$ as desired.

We have now seen that thanks to Axiom 4, embeddings have trivial kernels. In fact, Axiom 4 allows us to conclude also the converse: 
\begin{itemize}
	\item[\lby{obsAC}] if a homomorphism has a trivial kernel, then it must be an embedding. 
\end{itemize}
Indeed, for such homomorphism $f,$ in its factorization $f=\iota_{\mathsf{Im}f}h\pi_{\mathsf{Ker}f}$ the homomorphism $\pi_{\mathsf{Ker}f}$ is an isomorphism, and so $f$ is an embedding composed with an isomorphism, which forces $f$ itself to be an embedding. Dually, 
\begin{itemize}
	\item[\lby{obsD}] projections are precisely those homomorphisms whose images are the largest subgroups in the codomains.
\end{itemize}
This has three straightforward but significant consequences: 
\begin{itemize}
	\item[\lby{obsW}] Composite of two embeddings/projections is an embedding/projection. 
	
	\item[\lby{obsT}] If a composite $f=em$ is an embedding, then so must be $m,$ and dually, if it is a projection, then so must be $e.$
	
	\item[\lby{obsR}] A homomorphism is both an embedding and a projection if and only if it is an isomorphism.
\end{itemize}
Thanks to Axiom 2 we are also able to conclude the following useful facts:
\begin{itemize}
	\item[\lby{obsF}] A homomorphism is an embedding if and only if the corresponding direct image map is injective and if and only if the corresponding inverse image map is surjective. Dually, a homomorphism is a projection if and only if the corresponding inverse image map is injective and if and only if the corresponding direct image map is surjective.
	
	\item[\lby{obsAB}] A homomorphism is an embedding if and only if the corresponding direct image map is a right inverse of the corresponding inverse image map. Dually, a homomorphism is a projection if and only if the corresponding inverse image map is a right inverse of corresponding direct image map.
\end{itemize}
This has the following useful consequence:
\begin{itemize}
	\item[\lby{obsU}] For an embedding $m$ corresponding to a conormal subgroup $S,$ and for any two subgroups $A$ and $B$ such that $A\vee B\subseteq S,$ we have: $m^{-1}(A\vee B)=m^{-1}A\vee m^{-1}B.$
\end{itemize}

\add{Let us} write $G$ for the largest subgroup of a group $G$ and $1$ for its smallest subgroup. \add{Then, according to convention (\ref{obsAF}), for a conormal subgroup $A$ of $G$,
$$(A/1)/1=A/1,$$
and for a normal subgroup $B$ of $G$,
$$(G/1)/B=G/B.$$
This means that these two types of quotients are particular cases of the following construction of a \emph{subquotient} (for the first one, we would have to use (\ref{obsE})): when a subgroup $B$ of $G$ is normal to a conormal subgroup $A,$ we write $A/B$ to denote
$$A/B=(A/1)/\iota_A^{-1}B.$$
The dual construction to subquotient will be written as $B\backslash A$ (it is defined when $A$ is conormal to $B$). 
In group theory this simply means that $B$ is a normal subgroup of $G$ and $A$ is any subgroup of $G$ larger than $B,$ and we have $B\backslash A=A/B.$ We call $A/B$ the \emph{quotient} of $A$ by $B,$ and $B\backslash A$ the \emph{coquotient} of $A$ by $B.$ As we just remarked, in group theory, a coquotient is a particular instance of a quotient.}

\begin{description}
	\item[Axiom 5] Join of any two normal subgroups of a group is normal and meet of any two conormal subgroups is conormal.
\end{description} 
This axiom has the following reformulation: 
\begin{itemize}
	\item[\lby{obsA}] normal subgroups are stable under direct images along projections and conormal subgroups are stable under inverse images along embeddings. 
\end{itemize}
Let us now explain why is Axiom 5 equivalent to (\ref{obsA}). Consider two normal subgroups $A$ and $B$ of a group $G.$ Then $A\vee B=\pi_B^{-1}\pi_BA$ by Axiom 2. (\ref{obsA}) will give that $\pi_B A$ is normal and since normal subgroups are stable under inverse images, this will imply that $A\vee B$ is normal. So, by duality, (\ref{obsA}) implies Axiom~5. The argument for the converse implication is slightly less straightforward. Let $N$ be a normal subgroup of a group $G$ and let $p:G\to H$ be a projection. If the join $N\vee\mathsf{Ker}p$ is normal, then it is the kernel of the associated projection $p'.$ Since evidently $\mathsf{Ker}p\subseteq N\vee\mathsf{Ker}p,$ it follows by the universal property of the projection $p$ that $p'=vp$ for some (unique) homomorphism $v.$ We then have
$$pN=pp^{-1}pN=p(N\vee\mathsf{Ker}p)=p\mathsf{Ker}p'=pp^{-1}\mathsf{Ker}v=\mathsf{Ker}v$$
thus showing that $pN$ is normal. Thus, by duality, Axiom 5 implies (\ref{obsA}).

An analogue of (\ref{obsA}) for the normality relation can be easily obtained:
\begin{itemize}
	\item[\lby{obsAA}] For any projection $p\colon X\to Y$ and subgroups $A$ and $B$ of $X,$ if $B\vartriangleleft A$ then $pB\vartriangleleft pA.$ 
\end{itemize} 

\section{Homomorphism induction and a universal isomorphism theorem}\label{SecA}

Both, the isomorphisms arising in the isomorphisms theorems and the so-called connecting homomorphisms arising in homological diagram lemmas, can be obtained by composing relations created from zigzags of homomorphisms. In this section we develop a technique for detecting and producing these isomorphisms/connecting homomorphisms from zigzags of homomorphisms in our axiomatic study. The following fact, and its dual, provide basis for this technique:

\begin{lemma}\label{lemA}
	Any two normal subgroups $N$ and $R$ of a group $G$ gives rise to a commutative diagram
	$$\lby{eqD}\quad\vcenter{\xymatrix@=20pt{ & Z &  \\ X\ar[ur]^-{x} & & Y\ar[ul]_-{y} \\ & G\ar[ul]^-{n}\ar[ur]_-{r}\ar[uu]|-{p} & }}$$
	where $n,$ $p$ and $r$ are projections associated to $N,$ $N\vee R$ and $R,$ respectively. Moreover, in this diagram $y^{-1}xS=rn^{-1}S$ for any subgroup $S$ of $X.$    
\end{lemma}  

\begin{proof}
	From the fact that $p$ is a projection and $p=yr$ it follows that $y$ is a projection as well. We then have:
	\begin{align*}
	y^{-1}xS & = y^{-1}xnn^{-1}S \\
	& = y^{-1}yrn^{-1}S\\
	& = rr^{-1}y^{-1}yrn^{-1}S\\
	& = rp^{-1}pn^{-1}S\\
	& = r(n^{-1}S\vee\mathsf{Ker}p)\\ 
	& = r(n^{-1}S\vee N\vee R)\\ 
	& = r(n^{-1}S\vee R)\\ 
	& = (rn^{-1}S)\vee rR\\ 
	& = rn^{-1}S
	\end{align*}
\end{proof}        

Consider a zigzag
$$\lby{diaB}\quad\xymatrix@C=2pc{
	X_0\ar@{-}[r]^-{f_1} & X_1 & X_2\ar@{-}[l]_-{f_2} \ar@{-}[r]^-{f_3} & X_3 & X_4\ar@{-}[l]_-{f_4} \ar@{.}[r]& X_{n-1} & X_n\ar@{-}[l]_-{f_n}
}$$ of homomorphisms, where we have deliberately left out arrowheads since for each homomorphism an arrowhead is allowed to appear either on the left or on the right. The objects appearing in the zigzag will be called \emph{nodes} of the zigzag. We say that a subgroup $T$ of $X_n$ is obtained from a subgroup $S$ of $X_0$ by \textit{chasing} it (forward) along the zigzag when
$$T=f_n^\ast\cdots f_4^\ast f_3^\ast f_2^\ast f_1 ^\ast S$$
where $f_i^\ast=f_i$ when the arrowhead at the $i$-th place appears on the right and $f_i^\ast=f_i^{-1}$ when the arrowhead at the $i$-th place appears on the left. When in a zigzag (\ref{diaB}) all arrows pointing to the left are isomorphisms, we say that the zigzag is \emph{collapsible} and the composite 
$$f_n^\ast \cdots f_4^\ast f_3^\ast f_2^\ast f_1^\ast$$
where each $f_i^\ast$ is as before (this time, $f_i^{-1}$ representing the inverse of the isomorphism $f_i$) is called the \emph{induced homomorphism} of the collapsible zigzag.

As a direct consequence of (\ref{obsB}), we have:
\begin{itemize}
	\item[\lby{obsRR}] For a collapsible zigzag, the direct image and inverse image maps corresponding to the induced homomorphism are given by chasing subgroups forward and backward, respectively, along the zigzag.
\end{itemize}
Consider a diagram,
$$\lby{diaA}\quad\vcenter{\xymatrix@C=2pc{ X_0\ar@{-}[r]^-{f_1}\ar[d]_{g_0} & X_1\ar[d]_{g_1} & X_2\ar[d]_{g_2}\ar@{-}[l]_-{f_2} \ar@{-}[r]^-{f_3} & X_3\ar[d]_{g_3} & X_4\ar[d]_{g_4}\ar@{-}[l]_-{f_4} \ar@{.}[r]& X_{n-1}\ar[d]_{g_{n-1}} & X_n\ar[d]^{g_n}\ar@{-}[l]_-{f_n} \\
		X'_0\ar@{-}[r]_-{f'_1} & X'_1 & X'_2\ar@{-}[l]^-{f'_2} \ar@{-}[r]_-{f'_3} & X'_3 & X'_4\ar@{-}[l]^-{f'_4} \ar@{.}[r]& X'_{n-1} & X'_n\ar@{-}[l]^-{f'_n}
}}$$
with the top and bottom rows being zigzags of homomorphisms. A square
$$\lby{diaC}\quad\vcenter{\xymatrix{ X_{i-1}\ar@{-}[r]^-{f_i}\ar[d]_-{g_{i-1}} & X_{i}\ar[d]^-{g_i} \\ X'_{i-1}\ar@{-}[r]_-{f'_i} & X'_{i} }}$$ in this diagram is said to be commutative when $g_if_i=f'_ig_{i-1}$ if both horizontal arrows are right-pointing, $g_{i-1}f_i=f'_ig_{i}$ if they are both left-pointing, $f'_ig_if_i=g_{i-1}$ if the top arrow is right-pointing and the bottom arrow is left-pointing and $f'_ig_{i-1}f_i=g_i$ if the top arrow is left-pointing and the bottom arrow is right-pointing. It is not difficult to see that 
\begin{itemize}
	\item[\lby{obsL}]
	when in a commutative diagram (\ref{diaC}) either the top or the bottom arrows are isomorphisms, replacing any of such arrows with their inverses (pointing in the opposite direction), still results in a commutative diagram.
\end{itemize} 
From this we can conclude the following:
\begin{itemize}
	\item[\lby{obsJ}]
	In a diagram (\ref{diaA}) where all squares are commutative, if both the top and the bottom rows are collapsible, then $g_nf=f'g_0,$ where $f$ is a homomorphism induced by the top zigzag and $f'$ is a homomorphism induced by the bottom zigzag. Moreover, if all homomorphisms $g_i$ are isomorphisms, then the top row is collapsible if and only if the bottom row is collapsible. 
\end{itemize}  
In particular, this implies that 
\begin{itemize}
	\item[\lby{obsI}]
	when in a diagram (\ref{diaA}) where all squares commute and the top and bottom rows are collapsible zigzags, if $g_0$ and $g_n$ are identity homomorphisms then both collapsible zigzags induce the same homomorphism. 
\end{itemize}  

A \emph{subquotient} is a zigzag whose right-pointing arrows are projections and left-pointing arrows are embeddings. We will use the term \emph{cosubquotient} for the dual notion. As we will soon see, there is a canonical way of constructing from a zigzag another zigzag which is a subquotient followed by a cosubquotient. Notice that a zigzag is a subquotient if and only if its \emph{opposite zigzag} (the original zigzag reflected horizontally) is a cosubquotient.

\begin{lemma}\label{lemC}
	The opposite zigzag of a subquotient is collapsible if and only if
	\begin{itemize}
		\item[\lby{obsH}] chasing the trivial subgroup of the final node backward along the zigzag results in the trivial subgroup of the initial node.
	\end{itemize}
\end{lemma}

\begin{proof}
	By definition, a zigzag is collapsible if and only if all left-pointing arrows in it are isomorphisms, so being collapsible in the opposite direction is having right-pointing arrows isomorphisms. In a zigzag where these arrows are projections, this is the same as to require (thanks to the dual of (\ref{obsD}) and (\ref{obsR})) that the right-pointing arrows have trivial kernels. If (\ref{obsH}) holds, then chasing the trivial subgroup backward along the zigzag, up to the node just before the first one, will still result in a trivial subgroup. Indeed, if the arrow between the first two nodes is left-pointing then it is an embedding, and so the result follows from (\ref{obsF}). If it is right-pointing, then it is a projection and then the result follows from (\ref{obsAB}). Iterating the same argument we lead to the conclusion that chasing the trivial subgroup of the final node backward along the zigzag results in the trivial subgroup at each node, and consequently,
	\begin{itemize}
		\item[\lby{obsG}] chasing the trivial subgroup from any node to the left to the adjacent node results in the trivial subgroup of that node.
	\end{itemize}
	So (\ref{obsH}) implies (\ref{obsG}). Moreover, these properties are in fact equivalent to each other (the converse implication is straightforward). At the same time, in a subquotient, (\ref{obsG}) is equivalent to the right-pointing arrows having trivial kernels. This completes the proof.  
\end{proof}

Next, we describe a process which builds from a zigzag of homomorphisms (the base of the diagram below) a ``pyramid'' of homomorphisms:
$$\lby{eqA}\quad\vcenter{\xymatrix@!=1pt{
		&& && && X^{n}_0\ar@{-}[dl]\ar@{-}[dr]
		\\&& && & X^{n-1}_{0}\ar@{-}[dr] & & X^{n}_{1}\ar@{-}[dl]\ar@{-}[dr]
		\\&&  && X^{4}_{0}\ar@{.}[ur]\ar@{-}[dl]\ar@{-}[dr] & & X^{n-1}_{1}\ar@{-}[dr] && X^{n}_{2}\ar@{-}[dl]\ar@{-}[dr]
		\\& & & X^{3}_{0}\ar@{-}[dr] & & X^{4}_{1} \ar@{.}[ur]\ar@{-}[dl] & & X^{n-1}_{2}\ar@{-}[dr] & &   X^{n}_{3}\ar@{-}[dl]\ar@{-}[dr]
		\\& & X^{2}_{0}\ar@{-}[dl]\ar@{-}[ur]\ar@{-}[dr] & & X^{3}_{1} & & X^{4}_{2}\ar@{.}[ur] \ar@{-}[ul]\ar@{-}[dl]\ar@{-}[dr] & & X^{n-1}_{3}\ar@{-}[dr] & & X^{n}_{4}\ar@{-}[dl]\ar@{.}[dr]\\
		& X^{1}_{0}\ar@{-}[dr] & & X^{2}_{1}\ar@{-}[ur]\ar@{-}[dl] & & X^{3}_{2}\ar@{-}[ul]\ar@{-}[dr] & & X^{4}_{3}\ar@{.}[ur]\ar@{-}[dl] & & X^{n-1}_{4}\ar@{..}[dr] & & X^{n}_{n-1}\ar@{-}[dl] \\
		X_0^0 \ar@{-}[rr] \ar@{-}[ur] && X_1^1 && X_2^2\ar@{-}[ul]\ar@{-}[ur]\ar@{-}[ll]\ar@{-}[rr] && X_3^3 && X_4^4\ar@{.}[ur]\ar@{-}[ul]\ar@{-}[ll] \ar@{.}[rr] && X_{n-1}^{n-1} && X_n^n\ar@{-}[ul]\ar@{-}[ll]
}}$$
We construct this pyramid level by level, as follows. 

The bottom triangles
$$\xymatrix@!=5pt{ & X^{i}_{i-1} & \\ X_{i-1}^{i-1}\ar@{-}[rr]\ar@{-}[ur] & & X_{i}^i\ar@{-}[ul] }$$
are formed according to the following displays, depending on the direction of the base arrow,
$$\xymatrix@!=5pt{ & X^{i}_{i-1}\ar[dr]^-{m} & \\ X_{i-1}^{i-1}\ar[rr]_-{f_i}\ar[ur]^-{e} & & X_{i}^i }\quad\quad\quad\quad\quad\quad \xymatrix@!=5pt{ & X^{i}_{i-1}\ar[dl]_-{m} & \\ X_{i-1}^{i-1} & & X_{i}^i\ar[ul]_-{e}\ar[ll]^-{f_i} }$$
where $m$ is an embedding associated to the image of $f_i$ while $e$ is a projection associated to the kernel of $f_i,$ and $f_i=me.$ The possibility of forming such triangles is given by Axiom 4. As we build the pyramid bottom-up, we show on the way that all arrows in the pyramid which point downwards are embeddings and all arrows pointing upwards are projections. We can see that this is indeed the case so far from our construction of the triangles in the base $$\xymatrix@!=1pt{
	& X^{1}_{0}\ar@{-}[dr] & & X^{2}_{1}\ar@{-}[dl] & & X^{3}_{2}\ar@{-}[dr] & & X^{4}_{3}\ar@{-}[dl] & & & & X^{n}_{n-1}\ar@{-}[dl] \\
	X_0^0 \ar@{-}[rr] \ar@{-}[ur] && X_1^1 && X_2^2\ar@{-}[ul]\ar@{-}[ur]\ar@{-}[ll]\ar@{-}[rr] && X_3^3 && X_4^4\ar@{-}[ul]\ar@{-}[ll] \ar@{.}[rr] && X_{n-1}^{n-1} && X_n^n\ar@{-}[ul]\ar@{-}[ll]
}$$
of the pyramid. Next, we describe how to build the top wedge of a diamond
$$\lby{eqB}\quad\vcenter{\xymatrix@!=5pt{ & X^{i+1}_{j-1}\ar@{-}[dl]\ar@{-}[dr]  & \\ X^{i}_{j-1} & & X^{i+1}_j \\ & X^i_j\ar@{-}[ul]\ar@{-}[ur] & }}$$ once the bottom wedge has been built. When the arrows in the bottom wedge point upwards, they are projections. One then builds the diamond of the type that occurs in Lemma~\ref{lemA}, where $N$ and $R$ are taken to be the kernels of these projections. We refer to these diamonds as the \emph{projection diamonds} of the pyramid. The next case we would like to consider are the \emph{embedding diamonds}. These are dual to projection diamonds and they are built from a base wedge having arrows pointing downwards. There are two further types of diamonds that we will need to build (ignore the dotted arrows for now):
$$\xymatrix@!=5pt{ & X^{i+1}_{j-1}\ar[dr]^-{m_2}  & \\ X^{i}_{j-1}\ar[dr]_-{m_1}\ar[ur]^-{e_2}\ar@{..>}[rr] & & X^{i+1}_j \\ & X^i_j\ar[ur]_-{e_1} & }\quad\quad\quad\quad\quad \xymatrix@!=5pt{ & X^{i+1}_{j-1}\ar[dl]_-{m_2}  & \\ X^{i}_{j-1} & & X^{i+1}_j\ar@{..>}[ll]\ar[ul]_-{e_2}\ar[dl]^-{m_1}\\ & X^i_j\ar[ul]^-{e_1} & }$$
These are built by composing the homomorphisms in the bottom wedge to create the indicated dotted arrows and then building a triangle over it in the same way as in the initial step. Notice that in each of the four cases, the newly created arrows still have the same property that the upward directed ones are projections and the downward directed ones are embeddings. Thus we can build the entire pyramid layer by layer until we reach the top.

The pyramid construction is unique up to an isomorphism in the following sense:  
\begin{itemize}
	\item[\lby{obsK}] Two pyramids built from the same zigzag admit isomorphisms between the corresponding nodes so that the diagram formed from the two pyramids and the isomorphisms is commutative.
\end{itemize}
This can be easily verified by checking that such isomorphisms arise at each step in the building process from the isomorphisms created at the previous step. 

A zigzag portion of the pyramid
$$\xymatrix{ X_{p_1}^{q_1}\ar@{-}[r] & X_{p_2}^{q_2}\ar@{-}[r] & X_{p_3}^{q_3}\ar@{-}[r] &X_{p_4}^{q_4}\ar@{..}[r] & X_{p_{n-1}}^{q_{n-1}}\ar@{-}[r] & X_{p_{n}}^{q_{n}} }$$
connecting two nodes in the pyramid is said to be \emph{horizontal} when $\{p_{i+1}-p_i,q_{i+1}-q_i\}=\{0,1\}$ for each $i\in\{1,\dots,n-1\}.$ It is said to be \emph{vertical} when $\{p_{i}-p_{i+1},q_{i+1}-q_{i}\}=\{0,1\}$ for each $i\in\{1,\dots,n-1\}.$ The horizontal zigzag joining $X_0^0$ with $X^n_n,$ which runs up along the left side of the triangular outline of the pyramid, and then down along its right side, will be called the \emph{principal horizontal zigzag} of the pyramid. The vertical zigzag that joins $X_0^0$ with $X_0^n$ will be called the the \emph{left principal vertical zigzag} of the pyramid, while the vertical zigzag that joins $X_n^n$ with $X_0^n,$ will be called the \emph{right principal vertical zigzag} of the pyramid. Note that 
\begin{itemize}
	\item[\lby{obsAE}] the vertical zigzags of a pyramid are subquotients. 
\end{itemize}   

Thanks to Lemma~\ref{lemA} and its dual, commutativity of the pyramid, the fact that in the pyramid the upward directed arrows are projections and the downward directed arrows are embeddings and (\ref{obsAB}), subgroup chasing in the pyramid behaves as follows:
\begin{itemize}
	\item[\lby{obsM}] Chasing a subgroup from one node to another backward or forward along a horizontal zigzag in the pyramid does not depend on the choice of the path. 
	\item[\lby{obsP}] Chasing a subgroup from one node to another downward via a vertical zigzag and then upward via the same zigzag back to the original node gives back the original subgroup.
	\item[\lby{obsO}] Chasing a trivial/largest subgroup upward along a vertical zigzag always results in a trivial/largest subgroup.
\end{itemize}  
And furthermore, by Lemma~\ref{lemC} (and its dual) and (\ref{obsAE}) we also have:
\begin{itemize}
	\item[\lbn] Chasing a trivial/largest subgroup downward/upward along a vertical zigzag results in a trivial/largest subgroup if and only if the zigzag is collapsible in the downward/upward direction.
\end{itemize}  

When the principal horizontal zigzag in the pyramid constructed from a given zigzag is collapsible, we will say that the given zigzag \emph{induces a homomorphism} and the homomorphism induced by the collapsible principal horizontal zigzag is called the \emph{induced homomorphism} of the given zigzag. In view of (\ref{obsK}), and thanks to (\ref{obsI}) and (\ref{obsJ}), we have:
\begin{itemize}
	\item[\lby{obsAD}] for a given zigzag, the induced homomorphism is unique when it exists (and in particular, it does not depend on the actual pyramid). 
\end{itemize}
Using this and (\ref{obsL}), it is not difficult to check that 
\begin{itemize}
	\item[\lbn]
	when a zigzag is collapsible it induces a homomorphism in the above sense, and moreover, the induced homomorphism is the same as the homomorphism induced by the zigzag as a collapsible zigzag, in the sense defined earlier. 
\end{itemize} 
For this, one must first show that when the base of the pyramid is collapsible, all arrows in the pyramid that point south-west or north-west will be isomorphisms. This will imply that the principal horizontal zigzag is collapsible. After this, one can use commutativity of the pyramid to show that the homomorphism induced by the latter collapsible zigzag is the same as the one induced by the original collapsible zigzag.   

\begin{theorem}[Homomorphism Induction Theorem]
	For a zigzag to induce a homomorphism it is necessary and sufficient that chasing the trivial subgroup of the initial node forward along the zigzag results in the trivial subgroup of the final node, and chasing the largest subgroup of the final node backward along the zigzag results in the largest subgroup of the initial node. Moreover, when a zigzag induces a homomorphism, the induced homomorphism is unique and the direct and inverse image maps of the induced homomorphism are given by chasing a subgroup forward and backward, respectively, along the zigzag.  
\end{theorem}

\begin{proof}
	To prove the first part of the theorem, apply (\ref{obsM}) to reduce chasing along the zigzag to chasing along the principal horizontal zigzag of the pyramid constructed from the zigzag. Then the result follows from (\ref{obsO}) and Lemma~\ref{lemC} together with its dual. The second part follows from (\ref{obsAD}), (\ref{obsM}) and (\ref{obsRR}).
\end{proof}

\begin{corollary}[Universal Isomorphism Theorem]\label{corA}
	If largest and smallest subgroups are preserved by chasing them from each end node of a zigzag to the opposite end node, then the zigzag induces an isomorphism between its end nodes. 
\end{corollary}

\begin{proof}
	This can be obtained directly from the homomorphism induction theorem, by applying Axioms~4 to the induced homomorphism (see (\ref{obsAC}) and (\ref{obsR})). Alternatively, apply the homomorphism induction theorem to the zigzag and its opposite zigzag. It is easy to see that these zigzags will induce morphisms which are inverses of each other.  
\end{proof}

We conclude this section by drawing links with group theory.
Consider a zigzag (\ref{diaB}) of ordinary group homomorphisms. Each homomorphism $f_i$ is a function and so it can be viewed as a relation, given by 
$$xfy\quad\Leftrightarrow\quad f(x)=y.$$
Let us write $f_i^\ast$ to represent the relation $f_i$ when the arrowhead at the $i$-th place in the zigzag appears on the right, and otherwise, to represent the opposite relation $f_i^\mathsf{op}.$ We will refer to the composite $$f_n^\ast\cdots f_4^\ast f_3^\ast f_2^\ast f_1 ^\ast$$
of relations as the relation induced by the zigzag.

\begin{theorem}\label{ThmC}
	A zigzag of ordinary group homomorphisms induces a homomorphism if and only if the induced relation is a function, and when this is the case, that function is the induced homomorphism.
\end{theorem} 

\begin{proof}
	First we would like to establish as a general fact that for any two horizontal zigzags joining one node of a pyramid with another, the induced relations coincide. It is easy to see why this is the case when these two zigzags consist only of arrows from the triangles in the base of the pyramid. Further up, the result reduces to proving that for each diamond (\ref{eqB}) occurring in the pyramid, the relation induced by its bottom wedge is the same as the relation induced by its top wedge. For embedding diamonds, this fact can be deduced from a similar fact about an arbitrary commutative diagram
	$$
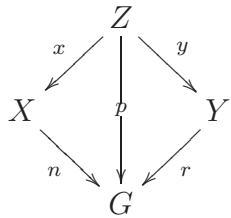
\xymatrix@=20pt{ & Z &  \\ X\ar@{<-}[ur]^-{x} & & Y\ar@{<-}[ul]_-{y} \\ & G\ar@{<-}[ul]^-{n}\ar@{<-}[ur]_-{r}\ar@{<-}[uu]|-{p} & }$$
	of injective functions between sets, where $\mathsf{Im}p=\mathsf{Im}n\wedge\mathsf{Im}r$ (this equality does hold in our case due to the construction of the embedding diamonds via the dual of Lemma~\ref{lemA}). For a projection diamond (\ref{eqD}), the fact that the relation induced by the bottom wedge is smaller than the relation induced by the top wedge follows straightforwardly from the commutativity of the diamond. To show that it is bigger, consider elements $a\in X$ and $b\in Y$ such that $x(a)=y(b).$ Since $n$ is a surjective map, we have $n(c)=a$ for some $c\in G.$ Now consider the element $br(c)^{-1}$ of $Y.$ This element belongs to the kernel of $y.$ Indeed, $$y(bn(c)^{-1})=y(b)(yr(c))^{-1}=y(b)(xn(c))^{-1}=y(b)x(a)^{-1}=1.$$ By Lemma~\ref{lemA}, we have the following equalities of subgroups:
	$$\mathsf{Ker}y=y^{-1}x1=rn^{-1}1=r\mathsf{Ker}n.$$
	Therefore, $br(c)^{-1}=r(d)$ for some element $d$ of $\mathsf{Ker}n.$ We are looking for an element $e$ of $G$ such that $n(e)=a$ and $r(e)=b.$ Take $e=dc.$ Then,
	$n(e)=n(d)n(c)=n(c)=a$ and $r(e)=r(d)r(c)=br(c)^{-1}r(c)=b.$ 
	
	Thus, we have proved that for any two horizontal zigzags joining one node of a pyramid (\ref{eqA}) with another, the induced relations coincide. In particular, this means that the relation induced by a zigzag of group homomorphisms is the same as the relation induced by the principal horizontal zigzag of the pyramid constructed from the original zigzag. If that zigzag is collapsible, then it is clear that the induced relation is precisely the induced homomorphism. Conversely, we can show that if the induced relation $R$ is a function then the zigzag in question is collapsible. Pick an element $a\in X_0^0.$ Since $R$ is a function, it must be related to some element in $X^n_n.$ This, in the case when the arrow joining $X_0^0$ with $X_0^1$ is left-pointing, will imply that $a$ belongs to the image of the left-pointing arrow. This would prove that this left-pointing arrow is surjective. However, the left-pointing arrows along the top left side of the pyramid are embeddings, and so we will get that this left-pointing arrow is an isomorphism. Since the right-pointing arrows along the top left side of the pyramid are projections and hence surjective maps, we can use a similar argument at the next step, to show that if the arrow connecting $X_0^1$ with $X_0^2$ is left-pointing, then it is an isomorphism. Continuing on like this, we would have proved that the top left side of the pyramid is a collapsible zigzag. A similar argument can be used to show that in the top right side of the pyramid all left-pointing arrows have trivial kernels and hence are isomorphisms.\end{proof}

\section{Classical isomorphism theorems}

In this section, after the following two lemmas which establish basic lattice properties of the normality relation, we show how the Universal Isomorphism Theorem (Corollary~\ref{corA}) allows one to deduce classical isomorphism theorems in our self-dual theory.  

\begin{lemma}\label{lemD}
	Let $G$ be any group and let $A,$ $B$ and $C$ be subgroups of $G$ such that $A\vartriangleleft B$ and $C$ is conormal. Then $A\wedge C\vartriangleleft B\wedge C.$ In particular, for $A\subseteq C\subseteq B,$ we obtain $A\vartriangleleft C.$ 
\end{lemma}

\begin{proof} Since both $B$ and $C$ are conormal, so is $B\wedge C$ by Axiom 5. By the universal property of $\iota_B,$ we have $\iota_Bi=\iota_{B\wedge C}$ for a unique \add{homo}morphism $i.$ Since $A\vartriangleleft B,$ we have $\iota_B^{-1}A$ is normal and therefore so is 
	$i^{-1}\iota_B^{-1}A=\iota_{B\wedge C}^{-1}A=\iota_{B\wedge C}^{-1}A\wedge\iota_{B\wedge C}^{-1}C=\iota_{B\wedge C}^{-1}(A\wedge C).$ This shows $A\wedge C\vartriangleleft B\wedge C.$ \end{proof}

\begin{lemma}\label{lemE}
	Let $A,$ $B$ and $C$ be subgroups of a group $G$ such that $A\vartriangleleft B$ and $C\vartriangleleft B\vee C.$ Then $A\vee C\vartriangleleft B\vee C.$ In particular, for $C\subseteq B$ the assumptions become $A\vartriangleleft B$ and $C\vartriangleleft B$ and they yield $A\vee C\vartriangleleft B.$  
\end{lemma}

\begin{proof}
	By (\ref{obsZ}) and Axiom~5, $\iota_{B\vee C}^{-1}A\vartriangleleft\iota_{B\vee C}^{-1}B.$ Since $C\vartriangleleft B\vee C,$ we have that $\iota_{B\vee C}^{-1}C$ is normal. By (\ref{obsAA}), $$\pi_{\iota_{B\vee C}^{-1}C}\iota_{B\vee C}^{-1}A\vartriangleleft \pi_{\iota_{B\vee C}^{-1}C}\iota_{B\vee C}^{-1}B$$ and since
	$$\pi_{\iota_{B\vee C}^{-1}C}^{-1}\pi_{\iota_{B\vee C}^{-1}C}\iota_{B\vee C}^{-1}B=\iota_{B\vee C}^{-1}B\vee \iota_{B\vee C}^{-1}C=(B\vee C)/1$$
	is conormal, by (\ref{obsZ}) we get
	$$\iota_{B\vee C}^{-1}(A\vee C)
	=\iota_{B\vee C}^{-1}A\vee \iota_{B\vee C}\vartriangleleft \pi_{\iota_{B\vee C}^{-1}C}^{-1}\pi_{\iota_{B\vee C}^{-1}C}\iota_{B\vee C}^{-1}B=(B\vee C)/1$$
	and so $A\vee C\vartriangleleft B\vee C,$ as desired.
\end{proof}

\begin{theorem}[Diamond Isomorphism Theorem] Consider two subgroups $A$ and $B$ of a group $G.$ If $B$ is conormal and $A\triangleleft A\vee B,$ then $A\wedge B\triangleleft B$ and there is an isomorphism $$B/(A\wedge B)\approx (A\vee B)/A.$$
	\label{2ist}
\end{theorem}
\begin{proof} 
	Suppose $B$ is conormal and $A\triangleleft A\vee B.$ Then by Lemma~\ref{lemD}, $A\wedge B\vartriangleleft (A\vee B)\wedge B=B.$ The required isomorphism will be induced by the zigzag
	$$\xymatrix@=35pt{B/(A\wedge B) & B/1\ar[l]_-{\pi_{\iota_{B}^{-1}(A\wedge B)}}\ar[r]^-{\iota_B} & G & (A\vee B)/1\ar[l]_-{\iota_{A\vee B}}\ar[r]^-{\pi_{\iota_{A\vee B}^{-1}A}} & (A\vee B)/A. }$$
\end{proof}

\begin{theorem}[Double-Quotient Isomorphism Theorem]\label{thmA} Let $N$ be a normal subgroup of a group $G.$ For any subgroup $S$ of $G/N,$ there exists a subgroup $R$ of $G$ having the following properties: (i) $N\subseteq R$ and $\pi_NR=S$; (ii) if $S$ is conormal, then $R$ is conormal to $N$ and $N\backslash R=S/1$; (iii) $S$ is normal if and only if $R$ is normal, and when this is the case, there is an isomorphism $G/R\approx (G/N)/S.$
	\label{3ist}
\end{theorem}
\begin{proof}
	Let $S$ be a subgroup of $G/N.$ By defining $R=\pi^{-1}_NS,$ we get (i), and also (ii) as an immediate consequence of (i). If $S$ is normal then $R$ is also normal by the dual of (\ref{obsV}). Conversely if $R$ is normal then $S$ is normal by (\ref{obsA}). The desired isomorphism is induced by the zigzag
	$$\xymatrix{G/R & G\ar[r]^-{\pi_N}\ar[l]_-{\pi_R} & G/N\ar[r]^-{\pi_S} & (G/N)/S.}$$
\end{proof}

The following isomorphism theorem from \cite{MB99} generalizes part (iii) of the Double Quotient Isomorphism Theorem: 

\begin{theorem}\label{thmB}
	Consider any group homomorphism $f\colon A \to B,$ and subgroups
	\[\mathsf{Ker} f \subseteq W \subseteq X\]
	of $A,$ where $X$ is conormal. Then $W \vartriangleleft X$ if and only if $fW \vartriangleleft fX,$ and when this is the case, there is an isomorphism
	$X/W \approx fX/fW.$
	\label{mac}
\end{theorem}
\begin{proof} Notice that since $X$ is a conormal subgroup of $A,$ $\mathsf{Im}(f\iota_X)=fX.$ Therefore, applying Axiom~4, we obtain a commutative diagram
	$$
	\xymatrix{ X/1\ar[r]^-{\iota_{X}}\ar[d]_-{\pi_{\mathsf{Ker}(f\iota_X)}} & A\ar[r]^{f} & B\\ (X/1)/\mathsf{Ker}(f\iota_X)\ar[rr]_-{h} & & fX/1\ar[u]_-{\iota_{fX}} \\
	}
	$$ where $h$ is an isomorphism.
	Suppose $W\vartriangleleft X.$ Then $\iota_X^{-1}W$ is a normal subgroup of $X/1.$ By (\ref{obsA}), $h\pi_{\mathsf{Ker}(f\iota_X)}\iota_X^{-1}W$ is a normal subgroup of $fX/1.$ The following calculation then completes the proof of $fW\vartriangleleft fX$: 
	\begin{align*}
	h\pi_{\mathsf{Ker}(f\iota_X)}\iota_X^{-1}W &= h\pi_{\mathsf{Ker}(f\iota_X)}\iota_X^{-1}(W\vee\mathsf{Ker}f)\\
	&= h\pi_{\mathsf{Ker}(f\iota_X)}\iota_X^{-1}f^{-1}fW\\
	&= h\pi_{\mathsf{Ker}(f\iota_X)}\pi_{\mathsf{Ker}(f\iota_X)}^{-1} h^{-1}\iota_{fX}^{-1} fW &\textrm{[ by (\ref{obsAB}) ]} \\
	&= hh^{-1}\iota_{fX}^{-1} fW\\
	&= \iota_{fX}^{-1} fW.
	\end{align*}
	Conversely, if $fW\vartriangleleft fX,$ then applying (\ref{obsZ}) we get $$W=W\vee\mathsf{Ker}f=f^{-1}fW\vartriangleleft f^{-1}fX=X\vee\mathsf{Ker}f=X$$
	since it is given that $X$ is conormal. The required isomorphism will be induced by the zigzag
	$$
	\xymatrix@=35pt{X/W & X/1\ar[l]_{\pi_{\iota^{-1}_XW}}\ar[r]^{\iota_X} & A\ar[r]^{f} & B & fX/1\ar[l]_{\iota_{fX}}\ar[r]^{\pi_{\iota^{-1}_{fX}fW}} & fX/fW.
	}
	$$	
\end{proof}

\begin{theorem}[Butterfly Lemma]
	Let $S'\vartriangleleft S$ and $T' \vartriangleleft T$ be conormal subgroups of a group $G$ \add{such that $S'\vee(S\wedge T)$ and $(S\wedge T)\vee T'$ are also conormal}. Then
	\begin{align*}
	S'\vee(S\wedge T') & \vartriangleleft S' \vee (S \wedge  T),\\
	(S'\wedge  T)\vee T' & \vartriangleleft (S \wedge  T)\vee T',
	\end{align*}
	and there is an isomorphism
	\[\frac{S'\vee(S\wedge  T)}{S'\vee(S\wedge  T')} \approx 
	\frac{(S\wedge  T)\vee T'}{(S'\wedge  T)\vee T'}.\]
\end{theorem}
\begin{proof} From the zigzag of subquotients determined by the normality relations $S'\vartriangleleft S$ and $T'\vartriangleleft T$, we build the pyramid ``downwards'' (for now, ignore the line above the pyramid):
	$$\xymatrix@r@!=15pt{ & \frac{S' \vee (S \wedge  T)}{S' \vee (S \wedge  T')} & & \frac{S' \vee (S \wedge  T)}{1}\ar@{-->}[ll]\ar@{-->}[dr] & & \frac{(S \wedge  T)\vee T'}{1}\ar[dl]\ar[rr] & & \frac{(S \wedge  T)\vee T'}{(S' \wedge  T)\vee T'} & \\ \displaystyle\frac{S}{S'}\ar@{=}[dr] & & \displaystyle\frac{S}{1}\ar@{=}[dr]\ar[dl]\ar[rr]^-{\iota_S}\ar[ll]_-{\pi_{\iota_S^{-1}S'}} & & G & & \displaystyle\frac{T}{1}\ar[rd]\ar@{=}[dl]\ar[ll]_-{\iota_T}\ar[rr]^-{\pi_{\iota_T^{-1}T'}} & & \displaystyle\frac{T}{T'}\ar@{=}[dl] \\ & \displaystyle\frac{S}{S'}\ar@{=}[dr] & & \displaystyle\frac{S}{1}\ar@{-->}[ur]\ar@{-->}[dl] & & \displaystyle\frac{T}{1}\ar[rd]\ar[ul]\ar[dr] & & \displaystyle\frac{T}{T'}\ar@{=}[dl] & \\ & & \displaystyle\frac{S}{S'} & & \displaystyle\frac{S\wedge T}{1}\ar@{..>}[ll]_-{f}\ar@{..>}[rr]^-{g}\ar[ul]\ar[ur]\ar@{..>}[uu]|-{\iota_{S\wedge T}}\ar@{..>}[dd]|-{\pi_{\mathsf{Ker}f\vee\mathsf{Ker}g}}\ar[dl]_-{\pi_{\mathsf{Ker}f}}\ar[dr]^-{\pi_{\mathsf{Ker}g}}\ar@{}[ll]\ar@{}[rr]  & & \displaystyle\frac{T}{T'} & & \\ & & & \displaystyle\frac{S\wedge T}{S'\wedge T}\ar@{-->}[dr]\ar@{-->}[ul] & & \displaystyle\frac{S\wedge T}{S\wedge T'}\ar[dl]\ar[ur] & & & \\ & & & & \displaystyle\frac{S\wedge T}{(S'\wedge T)\vee(S\wedge T')} & & & & }$$
	That the codomains of the bottom three $\pi$'s are the indicated quotients can be easily verified by chasing subgroups. This verification also gives $(S\wedge  T')\vee (S'\wedge  T)\vartriangleleft S \wedge  T$ and $S\wedge T'\vartriangleleft S\wedge T$. By Lemma~\ref{lemD}, $S'\vartriangleleft S$ implies $S'\vartriangleleft S'\vee(S\wedge T)$ and then by Lemma~\ref{lemE}, $S'\vee(S\wedge T') \vartriangleleft S' \vee (S \wedge  T)$. Symmetrically one has $S'\wedge T\vartriangleleft S\wedge T$ and $T'\vee(S'\wedge T) \vartriangleleft T' \vee (S \wedge  T).$ Next, we form the subquotients arising from these normality relations (the antennae of the butterfly). The zigzag indicated by the dashed arrows and its horizontal reflection will then induce isomorphisms 
	$$\frac{S'\vee(S\wedge  T)}{S'\vee(S\wedge  T')} \approx \frac{S\wedge  T}{(S'\wedge  T)\vee(S\wedge  T')}\approx 
	\frac{T'\vee(S\wedge  T)}{T'\vee(S'\wedge  T)}.$$
\end{proof}

\section{Concluding remarks} 

It is not difficult to verify that the theory presented in this paper will work just as well for other group-like structures such as rings without identity, algebras, modules and \add{more generally,} $\Omega$-groups, by replacing subgroups with relevant substructures (subrings in the case of rings, submodules in the case of modules, etc.) and group homomorphisms with relevant structure-preserving maps (ring homomorphisms, module homomorphisms, etc.). More generally, objects, subobjects and morphisms in an arbitrary semi-abelian category in the sense of \cite{JMT02} would give rise to a model for our theory. \add{Among contexts covered by our theory and not covered by the theory of semi-abelian categories is the context of rings with identity. For these, seen as ``groups'' in our theory, one must take additive subgroups as ``subgroups''. Then, ``normal subgroups'' will be ideals and ``conormal subgroups'' will be subrings. Our homomorphism theorems would recover standard homomorphism theorems for rings with identity, which rely on interplay between subrings and ideals (see e.g.~\cite{J85}).}

\add{ An attempt to analyze what breaks duality in a semi-abelian category would take one very close to Mac~Lane's remark from \cite{M50} that for non-abelian groups,}
\begin{itemize}
	\item[]\add{\emph{The primitive concepts of a category are not sufficient to formulate all the duality phenomena, and in particular do not provide for ``subgroups versus quotient groups,'' or ``homomorphisms onto versus isomorphisms into.''}}
\end{itemize}
\add{Restricting to the case of groups alone and ignoring other group-like structures, the obstacle pointed out above is actually avoidable, since ``isomorphisms into'' (i.e.~injective group homomorphisms) are the same as monomorphisms and ``homomorphisms onto'' are the same as epimorphisms (see e.g.~Exercise~5 in Section~5 of Chapter~I in \cite{M98}). This is precisely the approach taken up by Wyler in \cite{W66}. However, in many group-like structures (such as rings, for instance), epimorphisms are not the same as surjective homomorphisms. One way to deal with this would be to artificially extend the language of a category to allow to speak about morphisms from a designated class $\mathcal{E}$ of epimorphisms and a designated class $\mathcal{M}$ of monomorphisms, and declare that the dual of a statement $f\in\mathcal{E}$ is the statement $f\in\mathcal{M}$ (this idea appears already in \cite{M50}). What we propose in this paper, if formalized in a category-theoretic language, gives a more conceptual solution to the problem, which is to extend the language of the category of groups with the language of the \emph{bifibration} of subgroups (a bifibration is a functor, which, together with its dual functor, is a fibration in the sense of Grothendieck \cite{G59}), with duality being the usual \emph{functorial duality} (i.e., duality of the elementary theory of a functor --- see page 32 in \cite{M98}).} \add{This bifibration is the projection functor $\mathfrak{p}\colon\mathbf{Grp}_{\mathbf{2}}\to\mathbf{Grp}$, where: $\mathbf{Grp}_{\mathbf{2}}$ is the category of ``pairs of groups'' in the sense of algebraic topology --- its objects are pairs $(G,S)$ where $G$ is a group and $S$ is its subgroup; $\mathbf{Grp}$ is the category of groups and the functor $\mathfrak{p}$ projects a pair $(G,S)$ to its first component $G$. Functorial duality refers to a notion of duality obtained by the process of switching to the dual functor $\mathfrak{p}^\mathsf{op}\colon\mathbf{Grp}_{\mathbf{2}}^\mathsf{op}\to\mathbf{Grp}^\mathsf{op}$ (similarly as ``categorical duality'' refers to switching to a dual category). In this setting, injective and surjective homomorphisms can be recovered, in a mutually dual way, as morphisms having (dual) universal properties with respect to the functor $\mathfrak{p}$. Furthermore, all of our axioms can be exhibited as self-dual axioms on the functor $\mathfrak{p}$. The developments in \cite{JW14, JW16, JW16b} should be sufficient to recover the details.} \add{Let us just mention here that semi-abelian categories can be recovered as domains of those bifibrations that in addition to the self-dual axioms also satisfy the following requirements: every subgroup is conormal, the codomain of the bifibration has finite products and sums, and $\pi_S$ exists for every $S$ and not only a normal one, as stipulated in Axiom~3. Of these, only the first one (conormality of all subgroups) would be a non-dual requirement, since it would imply that $\iota_S$ exists also for every $S$ (while the requirement of the existence of finite products and sums in the codomain of the bifibration is clearly a self-dual requirement).}

\add{
In the case when $\iota_S$ and $\pi_S$ are always defined, or equivalently, when every subgroup has an associated embedding and an associated projection, deserves special attention. Firstly, in this case we have:}
\begin{itemize}
	\item[\lbn] \add{Join of conormal subgroups is conormal. Dually, meet of normal subgroups is normal.} 
\end{itemize}
\add{Indeed, consider two conormal subgroups $A$ and $B$ of a group $G$. Then $\iota_A=\iota_{A\vee B}u$ for some (unique) homomorphism $u$, and therefore $A=\mathsf{Im}\iota_A\subseteq\mathsf{Im}\iota_{A\vee B}$. Similarly, $B\subseteq\mathsf{Im}\iota_{A\vee B}$ and so $A\vee B\subseteq \mathsf{Im}\iota_{A\vee B}$, which gives $A\vee B=\mathsf{Im}\iota_{A\vee B}$, concluding the proof. This would simplify formulation of the Butterfly Lemma slightly, as we would not have to ask the joins $S'\vee(S\wedge T)$ and $(S\wedge T)\vee T'$ to be conormal in the statement of the theorem. More interestingly, under the requirement that every subgroup has an associated embedding and an associated projection, the bifibration of subgroups becomes part of a series of five adjoint functors (as noted in \cite{W14}), and these in fact fit in a $2$-dimensional model of a truncated simplicial set (i.e.~a model of the dual of the $2$-category of ordinals --- see \cite{S80}):}

$$\add{\xymatrix@=50pt{ \mathbf{Grp_2}\ar[d]|-{\mathfrak{p}}\ar@/^30pt/[d]|-{\mathfrak{g}}\ar@/_30pt/[d]|-{\mathfrak{q}}\\ \mathbf{Grp}\ar@/^15pt/[u]|-{\mathfrak{s}}\ar@{}@<8pt>[u]|{\dashv}\ar@/_15pt/[u]|-{\mathfrak{t}}\ar@{}@<-8pt>[u]|{\dashv}\ar@{}@<-22pt>[u]|{\dashv}\ar@{}@<+22pt>[u]|{\dashv}\ar@/^15pt/[d]|-{\mathfrak{d}}\ar@{}@<-8pt>[d]|{\dashv}\ar@{}@<8pt>[d]|{\dashv}\ar@/_15pt/[d]|-{\mathfrak{c}} \\ \mathbf{Grp_0}\ar[u]|-{\mathfrak{i}} }}$$ 
\add{Here $\mathbf{Grp}_{\mathbf{0}}$ is the category of ``trivial groups'' --- i.e., groups $G$ for which the largest subgroup $G$ coincides with the smallest subgroup $1$ (in all standard examples this category is equivalent to a category with exactly one  morphism). The functor $\mathfrak{i}$ is a subcategory inclusion functor. The functors $\mathfrak{s}$ and $\mathfrak{t}$ assign to a group $G$ the pairs $(G,1)$ and $(G,G)$, respectively. The functors $\mathfrak{q}$ and $\mathfrak{g}$ assign to a pair $(G,S)$ the codomain of $\pi_S$ and the domain of $\iota_S$, respectively (these would now be defined for all subgroups $S$). The component of the unit of the adjunction $\mathfrak{q}\dashv \mathfrak{s}$ at a pair $(G,S)$ is mapped to $\pi_S$ by $\mathfrak{p}$, while the component of the counit of the adjunction $\mathfrak{t}\dashv \mathfrak{g}$ at a pair $(G,S)$ is mapped to $\iota_S$ by $\mathfrak{p}$.}

Our theory only covers the case when $\mathfrak{p}$ is a faithful \add{(and }amnestic\add{)} functor (called a ``form'' in \cite{JW14, JW16, JW16b}). In all of the concrete examples of interest, fibres of $\mathfrak{p}$ are complete lattices, and so in those cases, since our axioms require $\mathfrak{p}$ to be a bifibration, $\mathfrak{p}$ becomes in fact a \emph{topological functor} (see \cite{Bru84} and the references there). It would be interesting to see how much of our theory can be \add{extended} to non-faithful functors and which new examples could this lead to. \add{In particular, the simplicial display above suggests to consider, as a non-faithful example for exploration, the case when $\mathfrak{p}$ is the composition of a category $\mathbb{C}$, viewed as a functor from the category of pairs of composable arrows of $\mathbb{C}$ to the category of arrows of $\mathbb{C}$.}  

Thus, the present paper shows that the difficulties in expressing duality phenomenon for group-like structures can be overcome by using functorial duality in the place of categorical duality. This idea originated in the second author's work \cite{J14} on the comparison between semi-abelian categories with those appearing in the work of Grandis in homological algebra \cite{G92,G12,G13}. The origins of the latter work go back to earlier investigations, among others that of Puppe \cite{P62}, on generalizations of abelian categories which get rid of the assumption of the existence of products. The presence of duality was always a strong feature in these developments, and it was also inherited in the work of Grandis. In \cite{BG07} a common generalization is given of Puppe exact categories (which are the same as Mitchell exact categories \cite{M65}) and semi-abelian categories, which, however, retains the non-dual character of semi-abelian categories. In \cite{J14} it was shown that a self-dual unification is possible through the functorial approach. If in a functorial structure as described above we impose the axioms considered in this paper together with the requirement that every ``subgroup'' is both normal and conormal, we \add{will in fact} arrive to the notion of a generalized exact category in the sense of Grandis. In this case, the entire structure is determined by the class of ``null morphisms'', i.e.~morphisms $f\colon X\to Y$ satisfying $\mathsf{Ker}f=X$ (or equivalently, $\mathsf{Im}f=1$). In the case of pointed categories, i.e.~when the category has a zero object and null morphisms are those that factor through a zero object, generalized exact categories are precisely the Puppe exact categories (which under the presence of binary products are just the abelian categories).

The classes of Puppe exact categories and its generalizations considered by Grandis are all closed under ``projective quotients'', i.e.~the quotients of categories where two morphisms are identified when they induce the same direct and inverse image maps between the designated subobject lattices (such quotient of the category of vector spaces over a given field, with subspace lattices as designated subobject lattices, gives the category of projective spaces over the same field). Our axiomatic context has the same feature. This allows one to add projective quotients of semi-abelian categories to the list of examples where our theory can be applied. 

\add{The string of five functors in the simplicial display above is not far from the idea of a ``pointed combinatorial exactness structure'' in the sense of \cite{JM03} (see also \cite{JM09}), although there the codomain of $\mathfrak{p}$ is not equipped with a category structure, and consecutively, the adjoint structure of this string is discarded. Similar five functors (with the adjoint structure) also appear in \cite{CJWW15}.}   

The method of chasing subgroups goes back to Mac Lane \cite{M63}, who developed it in the context of abelian categories and used it for proving diagram lemmas of homological algebra. These lemmas can also be established in our self-dual context. The techniques for proving them are essentially identical to those used in \cite{G13}, except for the Snake Lemma, which will require our Homomorphism Induction Theorem for constructing the connecting homomorphism (the proof of the Snake Lemma using this technique was first presented in the talk of the second author given at Category Theory 2015 in Aveiro, Portugal). Revisiting further aspects of non-abelian homological algebra, as developed by Grandis in \cite{G12,G13}, would be interesting. It should also be mentioned that the importance of the direct and inverse image maps between lattices of subobjects is prominent in Wyler's work \cite{W71}, while the power of this technique in the study of non-additive abelian categories (where lattices of subobjects are replaced with lattices of normal subobjects) was realized in the work of Grandis and goes back to \cite{G84}. In particular, the formula $$ff^{-1}A=A\vee\mathsf{Ker}f$$
from our Axiom 2, plays a crucial role both in the work of Wyler and Grandis, as it does in ours. As explained in \cite{J14}, this formula is also a key ingredient of the notion of a semi-abelian category, which encodes protomodularity introduced by Bourn in \cite{B91}. In universal algebra, this formula is equivalent to the statement that a subalgebra which contains the $0$-class of a congruence is a union of congruence classes. As first shown in \cite{B79}, this is a reformulation of Ursini's condition on a variety of universal algebras from \cite{U73}, which in the case of pointed varieties determine precisely the semi-abelian ones. Note that the dual of the above formula, 
$$f^{-1}fB=B\wedge\mathsf{Im}f,$$
holds in any variety of universal algebras, and more generally, in any regular category (and hence also in any topos).

Going a bit back in time, we must mention that since the discovery of the notion of a lattice and the modular law by Dedekind \cite{D97}, lattice-theoretic techniques have brought a great influence on the development of abstract algebra. An element-free unified approach in the study of group-like structures, based on substructure lattices, has been very explicitly proposed already by Ore \cite{O35,O36}, ten years before the subject of category theory was born. However, Ore restricts to considering only those substructures which form modular lattices, and hence his work is rather a forerunner for the work of Grandis, where in the strongest setting (that of a generalized exact category), subobjects being considered always form a modular lattice. Also, Ore's approach is more primitive in that he does not make use of the Galois connections (formally introduced by Ore only nine years later in \cite{O44}), which provide a fundamental tool in the work of Grandis, similarly as in our more general theory. The work of Ore did not seem to give rise to any immediate further development towards axiomatic unification of the study of group-like structures. Instead, relationships between the structure of a group and the structure of its subgroup lattice have been extensively explored --- see \cite{S94} and the reference there. It should be mentioned also that traces of the idea that a unified treatment of isomorphism theorems for group-like structures, which does not make use of the group operations should exist, can be found already in Noether's work \cite{N27}.  

It is expected that many results on group homomorphisms and subgroups can be captured by our theory. For instance, the Butterfly Lemma opens a way to the Jordan-H\"older Theorem, as it does in group theory, following Schreier \cite{S28} and Zassenhaus \cite{Z34}. Dual of the Jordan-H\"older theorem would then be a similar theorem for chief series in the place of composition series.

Duals of the isomorphism theorems established in the present paper correspond, in most cases, to the same isomorphism theorems in group theory, but with stronger assumptions. For instance, the dual of the Diamond Isomorphism Theorem in group theory will give: we have an isomorphism
$$B/(A\wedge B)\approx (A\vee B)/A$$
when $A$ and $A\wedge B$ are normal subgroups of a group $G$. Thus, these more restricted versions of isomorphism theorems are in fact their equivalent reformulations from the point of view of our theory! Duality would be a very practical tool in establishing diagram lemmas in our theory, as it would allow to prove the entire lemma by only proving its dual half, just as in the work of Grandis, and more classically, as in the context of an abelian category.    

\add{S}imilarly to what happens in group theory, we can establish two \add{restricted versions of the modular law for subgroups (recall that for ordinary subgroups, the modular law holds only when certain subgroups are assumed to be normal)}, and \add{in our theory} these two modular laws happen to be dual to each other:   

\begin{lemma}[Restricted Modular Law]\label{lemB}
	For any three subgroups $X,$ $Y,$ and $Z$ of a group $G,$ if either $Y$ is normal and $Z$ is conormal, or $Y$ is conormal and $X$ is normal, then we have:
	$$X\subseteq Z \quad \implies \quad X\vee(Y\wedge Z)= (X\vee Y)\wedge Z.$$
	\label{l1}
\end{lemma}
\begin{proof}
	The second case, when $Y$ is conormal and $X$ is normal, is dual to the first case, when $Y$ is normal and $Z$ is conormal, so it suffices to prove the implication in the first case. 
	\add{A}ssume $Y= \mathsf{Ker}g$ and $Z =\mathsf{Im}f$ for some homomorphisms $g$ and $f$ respectively. We must prove
	\[X\subseteq \mathsf{Im}f \qquad \implies \qquad  X\vee(\mathsf{Ker}g\wedge \mathsf{Im}f)=(X\vee \mathsf{Ker}g)\wedge \mathsf{Im}f.\]
	Indeed, if $X \subseteq \text{Im}\,f$ then
	\begin{align*}
	X\vee (\mathsf{Ker}g\wedge \mathsf{Im}f) &= (X\wedge \mathsf{Im}f)\vee (\mathsf{Ker}g\wedge \mathsf{Im}f)\\
	&=ff^{-1}X\vee ff^{-1}\mathsf{Ker}g\\
	&= f(f^{-1}X\vee f^{-1}\mathsf{Ker}g)\\
	&= f(f^{-1}X\vee \mathsf{Ker}gf)\\
	&= f(gf)^{-1}gf(f^{-1}X)\\
	&= f(gf)^{-1}g(ff^{-1}X)\\
	&= f(gf)^{-1}g(X\wedge \mathsf{Im}f)\\
	&= f(gf)^{-1} gX\\
	&= ff^{-1}g^{-1}gX\\
	&= ff^{-1}(X \vee \mathsf{Ker}g)\\
	&= (X \vee \mathsf{Ker}g)\wedge \mathsf{Im}f.
	\end{align*}
\end{proof}

Had we followed the standard proof of the Butterfly Lemma, we would have needed to use the above Restricted Modular Law along with the Diamond Isomorphism Theorem. 

Our ``embedding diamonds'' are in fact pullbacks and our ``projection diamonds'' are pushouts in the sense of category theory. Subquotients in group theory (and more generally, in any semi-abelian category, as well as in the work of Grandis) can always be reduced to subquotients of length two (a projection followed by an embedding), thanks to the fact that there a pullback of a projection along an embedding is a projection. The dual of this is false (a pushout of an embedding along a projection is not necessarily an embedding), and so we cannot require its validity in our theory. Because of this, we are not able to define and compose relations in the ``usual way''. However, the pyramid should still allow defining relations in our theory as suitable equivalence classes of zigzags, in a way which in the case of each particular type of group-like structures, recaptures precisely the homomorphic relations. This would not be too surprising --- a very particular instance of our pyramid already appears implicitly in the way Grandis defines composition of relations in his contexts, which in fact goes back to the work of Tsalenko \cite{T67}. A thorough investigation of this topic as well as the closely related topic of coherence theorems for connecting homomorphisms should be interesting to carry out.

\bigskip 

Several talks have been given by the second author on different aspects of the subject of the present paper, which includes his invited talk at the international category theory conference, Category Theory 2013, held at Macquarie University in Sydney, where the validity of the isomorphism theorems in the axiomatic context presented in this paper was first announced. The direct proof of the Butterfly Lemma using the Universal Isomorphism Theorem was obtained during the visit of the first author at Stellenbosch University in October 2016.  

\section*{Acknowledgement}

\add{The authors would like to express deep thanks to the anonymous referee for his/her useful remarks on the earlier version of the paper.}


\end{document}